\documentclass[12pt]{amsart}
\usepackage{amsmath,amsthm,amsfonts,amssymb,mathrsfs}
\date{\today}

\usepackage{color}

\newtheorem{theorem}{Theorem}[section]
\newtheorem{proposition}[theorem]{Proposition}
\newtheorem{corollary}[theorem]{Corollary}
\newtheorem{lemma}[theorem]{Lemma}

\theoremstyle{definition}

\newtheorem{remark}[theorem]{Remark}
\newtheorem{definition}[theorem]{Definition}

\usepackage{hyperref}
\usepackage{multind}

\begin{document}

\title[Pseudocompact primitive topological inverse semigroups]{Pseudocompact primitive topological inverse semigroups}

\author{Oleg~Gutik}
\address{Department of Mathematics, Ivan Franko Lviv National
University, Universytetska 1, Lviv, 79000, Ukraine}
\email{o\underline{\hskip5pt}\,gutik@franko.lviv.ua,
ovgutik@yahoo.com}

\author{Kateryna~Pavlyk}
\address{Institute of Mathematics,
University of Tartu, J. Liivi 2, 50409, Tartu, Estonia}
\email{kateryna.pavlyk@ut.ee}

\keywords{Topological semigroup, topological inverse semigroup,
primitive inverse semigroup, Brand semigroup, Brandt
$\lambda$-extension, topological Brandt $\lambda$-extension, topological group}

\subjclass[2000]{22A15, 54G12, 54H10, 54H12}

\begin{abstract}
In the paper we study pseudocompact primitive topological inverse semigroups. We describe the structure of pseudocompact primitive topological inverse semigroups and show that the Tychonoff product of a family of pseudocompact primitive topological inverse semigroups is a pseudocompact topological space. Also we prove that the Stone-\v{C}ech compactification of a pseudocompact primitive topological inverse semigroup is a compact primitive topological inverse semigroup.
\end{abstract}

\maketitle

\section{Introduction and preliminaries}

Further we shall follow the terminology of \cite{BucurDeleanu1968, CHK, CP,
Engelking1989, Petrich1984}. The set of positive integers is denoted by $\mathbb{N}$.

A semigroup is a non-empty set with a binary associative
operation. A semigroup $S$ is called \emph{inverse} if for any
$x\in S$ there exists a unique $y\in S$ such that $x\cdot y\cdot
x=x$ and $y\cdot x\cdot y=y$. Such an element $y$ in $S$ is called
\emph{inverse} of $x$ and denoted by $x^{-1}$. The map defined on
an inverse semigroup $S$ which maps any element $x$ of $S$ its
inverse $x^{-1}$ is called the \emph{inversion}.

If $S$ is a semigroup, then by $E(S)$ we denote the subset of idempotents of
$S$, and by $S^1$ (resp., $S^0$) we denote the semigroup $S$ with the
adjoined unit (resp., zero). Also if a
semigroup $S$ has zero $0_S$, then for any $A\subseteq S$ we
denote $A^*=A\setminus\{ 0_S\}$.

If $E$ is a semilattice, then the semilattice operation on $E$
determines the partial order $\leqslant$ on $E$: $$e\leqslant
f\quad\text{if and only if}\quad ef=fe=e.$$ This order is called
{\em natural}. An element $e$ of a partially ordered set $X$ is
called {\em minimal} if $f\leqslant e$  implies $f=e$ for $f\in
X$. An idempotent $e$ of a semigroup $S$ without zero (with zero)
is called \emph{primitive} if $e$ is a minimal element in $E(S)$
(in $(E(S))^*$).

Let $S$ be a semigroup with zero and $\lambda$ be a cardinal $\geqslant 1$. On the set
$B_{\lambda}(S)=\left(\lambda\times S\times\lambda\right)\sqcup\{ 0\}$ we define the semigroup operation as follows
 $$
 (\alpha,a,\beta)\cdot(\gamma, b, \delta)=
  \begin{cases}
    (\alpha, ab, \delta), & \text{ if } \beta=\gamma; \\
    0, & \text{ if } \beta\ne \gamma,
  \end{cases}
 $$
and $(\alpha, a, \beta)\cdot 0=0\cdot(\alpha, a, \beta)=0\cdot
0=0,$ for all $\alpha, \beta, \gamma, \delta\in \lambda$ and $a,
b\in S$. If $S=S^1$ then the semigroup $B_\lambda(S)$ is called
the {\it Brandt $\lambda$-extension of the semigroup}
$S$~\cite{GutikPavlyk2001}. Obviously, ${\mathcal J}=\{
0\}\cup\{(\alpha, {\mathscr O}, \beta)\mid {\mathscr O}$ is the
zero of $S\}$ is an ideal of $B_\lambda(S)$. We put
$B^0_\lambda(S)=B_\lambda(S)/{\mathcal J}$ and we shall call
$B^0_\lambda(S)$ the {\it Brandt $\lambda^0$-extension of the
semigroup $S$ with zero}~\cite{GutikPavlyk2006}. Further, if
$A\subseteq S$ then we shall denote $A_{\alpha,\beta}=\{(\alpha,
s, \beta)\colon s\in A \}$ if $A$ does not contain zero, and
$A_{\alpha,\beta}=\{(\alpha, s, \beta)\colon s\in A\setminus\{ 0\}
\}\cup \{ 0\}$ if $0\in A$, for $\alpha, \beta\in {\lambda}$. If
$\mathcal{I}$ is a trivial semigroup (i.e., $\mathcal{I}$ contains only one element), then by ${\mathcal{I}}^0$ we denote the semigroup $\mathcal{I}$ with the adjoined zero. Obviously, for any
$\lambda\geqslant 2$ the Brandt $\lambda^0$-extension of the
semigroup ${\mathcal{I}}^0$ is isomorphic to the semigroup of
$\lambda\times\lambda$-matrix units and any Brandt
$\lambda^0$-extension of a semigroup with zero contains the
semigroup of $\lambda\times\lambda$-matrix units. Further by
$B_\lambda$ we shall denote the semigroup of $\lambda\times\lambda$-matrix units and by $B^0_\lambda(1)$ the subsemigroup of $\lambda\times\lambda$-matrix units of the Brandt
$\lambda^0$-extension of a monoid $S$ with zero. A completely
$0$-simple inverse semigroup is called a \emph{Brandt
semigroup}~\cite{Petrich1984}. By Theorem~II.3.5~\cite{Petrich1984}, a semigroup $S$ is a Brandt semigroup if and only if $S$ is isomorphic to a Brandt $\lambda$-extension $B_\lambda(G)$ of some group $G$.

Let $\left\{S_{\iota}\colon \iota\in\mathscr{I}\right\}$ be a disjoint family of semigroups with zero such that $0_\iota$ is zero in $S_{\iota}$ for any $\iota\in\mathscr{I}$. We put $S=\{0\}\cup\bigcup\left\{S_{\iota}^*\colon \iota\in\mathscr{I}\right\}$, where $0\notin\bigcup\left\{S_{\iota}^*\colon \iota\in\mathscr{I}\right\}$, and define a semigroup operation on $S$ in the following way
\begin{equation*}
    s\cdot t=
\left\{
  \begin{array}{cl}
    st, & \hbox{if } st\in S_{\iota}^* \hbox{ for some } \iota\in\mathscr{I};\\
    0, & \hbox{otherwise}.
  \end{array}
\right.
\end{equation*}
The semigroup $S$ with such defined operation is called the \emph{orthogonal sum} of of the family of semigroups $\left\{S_{\iota}\colon \iota\in\mathscr{I}\right\}$ and in this case we shall write $S=\sum_{\iota\in\mathscr{I}}S_{\iota}$.

A non-trivial inverse semigroup is called a \emph{primitive inverse semigroup} if all its non-zero idempotents are primitive~\cite{Petrich1984}. A semigroup $S$ is a primitive inverse semigroup if and only if $S$ is the orthogonal sum of the family of Brandt semigroups~\cite[Theorem~II.4.3]{Petrich1984}. We shall call a Brandt subsemigroup $T$ of a primitive inverse semigroup $S$ \emph{maximal} if every Brandt subsemigroup of $S$ which contains $T$, coincides with $T$.

Green's relations $\mathscr{L}$, $\mathscr{R}$ and $\mathscr{H}$ on a semigroup $S$ are defined by:
\begin{itemize}
    \item[] $a\mathscr{L}b$ \; if and only if \; $\{a\}\cup Sa=\{b\}\cup Sb$;
    \item[] $a\mathscr{R}b$ \; if and only if \; $\{a\}\cup aS=\{b\}\cup bS$;
          \;  and
    \item[] $\mathscr{H}=\mathscr{L}\cap\mathscr{R}$,
\end{itemize}
for $a, b\in S$. For details about Green's relations see \cite[\S~2.1]{CP} or \cite{Green1951}. We observe that two non-zero elements $(\alpha_1,s,\beta_1)$ and
$(\alpha_2,t,\beta_2)$ of a Brandt semigroup $B_\lambda(G)$, $s,t\in G$, $\alpha_1,\alpha_2,\beta_1,\beta_2\in\lambda$, are $\mathscr{ H}$-equivalent if and only if $\alpha_1=\alpha_2$ and $\beta_1=\beta_2$ (see \cite[p.~93]{Petrich1984}).

In this paper all topological spaces are Hausdorff. If $Y$ is a subspace of a
topological space $X$ and $A\subseteq Y$, then by
$\operatorname{cl}_Y(A)$ we denote the topological closure of $A$
in $Y$.

We recall that a topological space $X$ is said to be
\begin{itemize}
  \item \emph{compact} if each open cover of $X$ has a finite subcover;
  \item \emph{countably compact} if each open countable cover of $X$ has a finite subcover;
  \item \emph{pseudocompact} if each locally finite open cover of $X$ is finite.
\end{itemize}
According to Theorem~3.10.22 of \cite{Engelking1989}, a Tychonoff topological space $X$ is pseudocompact if and only if each continuous real-valued function on $X$ is bounded. Also, a Hausdorff topological space $X$ is pseudocompact if and only if every locally finite family of non-empty open subsets of $X$ is finite. Every compact space and every countably compact space are  pseudocompact (see \cite{Engelking1989}).

We recall that the Stone-\v{C}ech compactification of a Tychonoff space $X$ is a
compact Hausdorff space $\beta X$ containing $X$ as a dense subspace so that each continuous map $f\colon X\rightarrow Y$ to a compact Hausdorff space $Y$ extends to a continuous map $\overline{f}\colon \beta X\rightarrow Y$ \cite{Engelking1989}.

A {\it topological semigroup} is a Hausdorff topological space with a continuous semigroup operation. A topological semigroup which is an inverse semigroup is called an \emph{inverse topological semigroup}. A \emph{topological inverse semigroup} is an inverse topological semigroup with continuous inversion. A {\it topological group} is a topological space with a continuous group operation and inversion. We observe that the inversion on a topological inverse semigroup is a homeomorphism (see \cite[Proposition~II.1]{EberhartSelden1969}). A Hausdorff topology $\tau$ on a (inverse) semigroup $S$ is called (\emph{inverse}) \emph{semigroup} if $(S,\tau)$ is a topological (inverse) semigroup.

\begin{definition}[\cite{GutikPavlyk2001}]\label{def1}
Let $\mathfrak{TSG}$ be some category of topological semigroups. Let $\lambda$ be a cardinal $\geqslant 1$ and $(S,\tau)\in\mathbf{Ob}\mathfrak{TSG}$ be a topological monoid. Let $\tau_{B}$ be a topology on $B_{\lambda}(S)$ such that
\begin{itemize}
    \item[a)] $\left(B_{\lambda}(S),\tau_{B}\right)\in\mathbf{Ob}\mathfrak{TSG}$; and
    \item[b)] for some $\alpha\in{\lambda}$ the topological subspace $(S_{\alpha,\alpha},\tau_{B}|_{S_{\alpha,\alpha}})$ is naturally homeomorphic to $(S,\tau)$.
\end{itemize}
Then $\left(B_{\lambda}(S), \tau_{B}\right)$ is called a {\it topological Brandt $\lambda$-extension of $(S, \tau)$ in $\mathfrak{TSG}$}.
\end{definition}

\begin{definition}[\cite{GutikPavlyk2006}]\label{def2}
Let $\mathfrak{TSG}_0$ be some category of topological semigroups with zero. Let $\lambda$ be a cardinal $\geqslant 1$ and $(S,\tau)\in\mathbf{Ob}\mathfrak{TSG}_0$. Let $\tau_{B}$ be a topology on $B^0_{\lambda}(S)$ such that
\begin{itemize}
  \item[a)] $\left(B^0_{\lambda}(S),             \tau_{B}\right)\in\mathbf{Ob}\mathfrak{TSG}_0$;
  \item[b)] for some $\alpha\in{\lambda}$ the topological subspace $(S_{\alpha,\alpha},\tau_{B}|_{S_{\alpha,\alpha}})$ is naturally homeomorphic to $(S,\tau)$.
\end{itemize}
Then $\left(B^0_{\lambda}(S), \tau_{B}\right)$ is called a {\it topological Brandt $\lambda^0$-extension of $(S, \tau)$ in $\mathfrak{TSG}_0$}.
\end{definition}

We observe that for any topological Brandt $\lambda$-extension
$B_\lambda(S)$ of a topological semigroup $S$ in the category of topological semigroups there exist a topological monoid $T$ with zero and a topological Brandt $\lambda^0$-extension $B^0_\lambda(T)$ of $T$ in the category of topological semigroups with zero, such that the semigroups $B_\lambda(S)$ and $B^0_\lambda(T)$ are topologically isomorphic. Algebraic properties of Brandt $\lambda^0$-extensions of monoids with zero, non-trivial homomorphisms between them, and a category which objects are ingredients of the construction of such extensions were described in \cite{GutikRepovs2010}. Also, in  \cite{GutikPavlykReiter2009} and \cite{GutikRepovs2010} a category which objects are ingredients in the constructions of finite (resp., compact, countably compact) topological Brandt $\lambda^0$-extensions of topological monoids with zeros was described.

Gutik and Repov\v{s} proved that any $0$-simple countably compact topological inverse semigroup is topologically isomorphic to a topological Brandt $\lambda$-extension $B_{\lambda}(H)$ of a countably compact topological group $H$ in the category of topological inverse semigroups for some finite cardinal $\lambda\geqslant 1$ \cite{GutikRepovs2007}.  Also, every $0$-simple pseudocompact topological inverse semigroup is topologically isomorphic to a topological Brandt $\lambda$-extension $B_{\lambda}(H)$ of a pseudocompact topological group $H$ in the category of topological inverse semigroups for some finite cardinal $\lambda\geqslant 1$ \cite{GutikPavlykReiter2011}. Next Gutik and Repov\v{s} showed in \cite{GutikRepovs2007} that the Stone-\v{C}ech compactification $\beta(T)$ of a $0$-simple countably compact topological inverse semigroup $T$ is a $0$-simple compact topological inverse semigroup. It was proved in \cite{GutikPavlykReiter2011} that the same is true in the case of $0$-simple pseudocompact topological inverse semigroups.

In the paper \cite{BerezovskiGutikPavlyk2010} the structure of compact and countably compact primitive topological inverse semigroups was described and showed that any countably compact primitive topological inverse semigroup embeds into a compact primitive topological inverse semigroup.

In this paper we describe the structure of pseudocompact primitive topological inverse semigroups and show that the Tychonoff product of a family of pseudocompact primitive topological inverse semigroups is a pseudocompact topological space. Also we prove that the Stone-\v{C}ech compactification of a pseudocompact primitive topological inverse semigroup is a compact primitive topological inverse semigroup.

\section{Primitive pseudocompact topological inverse semigroups}

\begin{proposition}\label{proposition-2.1}
Let $S$ be a Hausdorff pseudocompact primitive topological inverse semigroup and $S$ be an orthogonal sum of the family $\{B_{\lambda_{i}}(G_i)\}_{i\in\mathscr{I}}$ of topological Brandt semigroups with zeros, i.e. $S=\sum_{i\in\mathscr{I}}B_{\lambda_{i}}(G_i)$. Then the following statements hold:
\begin{itemize}
    \item[$(i)$]  every non-zero idempotent of $S$ is an isolated point in $E(S)$ and $E(S)$ is a compact semilattice;

    \item[$(ii)$] every non-zero $\mathscr{H}$-class in $S$ is a pseudocompact closed-and-open subset of $S$;

    \item[$(iii)$] every maximal subgroup in $S$ is a pseudocompact subspace of $S$;

    \item[$(iv)$] every maximal Brandt subsemigroup of $S$ is a pseudocompact space and has finitely many idempotents.
\end{itemize}
\end{proposition}

\begin{proof}
$(i)$  First part of the statement follows from Lemma~7~\cite{BerezovskiGutikPavlyk2010}. Then the continuity of the semigroup operation and inversion in $S$ implies that the map $\mathfrak{e}\colon S\rightarrow E(S)$ defined by the formula $\mathfrak{e}(x)=x\cdot x^{-1}$ is continuous and hence by Theorem~3.10.24~\cite{Engelking1989}, $E(S)$ is a pseudocompact subspace of $S$ such that every non-zero idempotent in $E(S)$ is an isolated point. Therefore $E(S)$ is compact. Otherwise there exists an open neighbourhood $U(0)$ of the zero $0$ of $S$ in $E(S)$ such that the set $E(S)\setminus U(0)$ is infinite. But this contradicts the pseudocompactness of $E(S)$.

$(ii)$ By Corollary~8 from \cite{BerezovskiGutikPavlyk2010} every non-zero $\mathscr{H}$-class in $S$ is a closed-and-open subset of $S$ and hence by Exercise~3.10.F(d) is pseudocompact.

Statement $(iii)$ follows from $(ii)$.

$(iv)$ Let $B_{\lambda_i}(G_i)$ be a maximal Brandt subsemigroup of the semigroup $S$. Then statement $(i)$ implies that $E(S)$ is a compact and since every non-zero idempotent of $S$ is an isolated point of $E(S)$ we conclude that $E(B_{\lambda_{i}}(G_i))$ is compact for every $i\in\mathscr{I}$. By Corollary~3.10.27 of \cite{Engelking1989} the product of a compact space and a pseudocompact space is a pseudocompact space, and hence we have that the space $E(B_{\lambda_{i}}(G_i))\times S$ is pseudocompact. Since $S$ is a primitive inverse semigroup we conclude that $B_{\lambda_{i}}(G_i)=E(B_{\lambda_{i}}(G_i))\cdot S$. Now, the continuity of the semigroup operation in $S$ implies that the map $\mathfrak{f}\colon E(B_{\lambda_{i}}(G_i))\times S\rightarrow S$ defined by the formula $\mathfrak{f}(e,s)=e\cdot s$ is continuous, and since the continuous image of a pseudocompact space is pseudocompact we conclude that $B_{\lambda_{i}}(G_i)$ is pseudocompact. The last statement follows from Theorem~1 of~\cite{GutikPavlykReiter2011}.
\end{proof}

\begin{lemma}\label{lemma-2.2}
Let $U$ be an open non-empty subset of a topological group $G$ and $A$ be a dense subset of $G$. Then $A\cdot U=U\cdot A=G$.
\end{lemma}

\begin{proof}
Since G is a topological group we have that there exists a nonempty open subset $V$ of $G$ such that $V^{-1}=U$. Let $x$ be an arbitrary point of $G$. Then $x\cdot V$ is a nonempty open subset of $G$, because translations in every topological group are homeomorphisms. Then we have that $x\cdot V\cap A\neq\varnothing$ and hence $x\in A\cdot V^{-1}=A\cdot U$. Therefore we get that $G\subseteq A\cdot U$. The converse inclusion is trivial. Hence $A\cdot U=G$. The proof of the equality $U\cdot A=G$ is similar.
\end{proof}

\begin{lemma}\label{lemma-2.3}
Let $\lambda\geqslant 2$ be any cardinal and $U$ be an open non-empty subset of a topological inverse Brandt semigroup $B_\lambda(G)$ such that $U\neq\{0\}$. Then $A\cdot U\cdot A=B_\lambda(G)$ for every dense subset $A$ of $B_\lambda(G)$.
\end{lemma}

\begin{proof}
By Lemma~7~\cite{BerezovskiGutikPavlyk2010} we have that every non-zero idempotent of the topological inverse semigroup $B_\lambda(G)$ is an isolated point in $E(B_\lambda(G))$. The continuity of the semigroup operation and inversion in $S$ implies that the map $\mathfrak{e}\colon S\rightarrow E(S)$ defined by the formula $\mathfrak{e}(x)=x\cdot x^{-1}$ is continuous and hence $G_{\alpha,\beta}$ is an open-and-closed subset of $B_\lambda(G)$ for all $\alpha,\beta\in\lambda$. Since $A$ is a dense subset of $B_\lambda(G)$ we conclude that $A\cap G_{\alpha,\beta}$ is a dense subset in $G_{\alpha,\beta}$ for all $\alpha,\beta\in\lambda$. Also, since $\lambda\geqslant 2$ we have that $0\in A\cdot U\cdot A$. This implies that it is sufficient to show that $G_{\alpha,\beta}\subseteq A\cdot U\cdot A$ for all $\alpha,\beta\in\lambda$.

Since $G_{\alpha,\beta}$ is an open subset of $B_\lambda(G)$ for all $\alpha,\beta\in\lambda$, without loss of generality we assume that $U\subseteq G_{\alpha_0,\beta_0}$ for some $\alpha_0,\beta_0\in\lambda$, i.e., $U=V_{\alpha_0,\beta_0}$ for some open subset $V\subseteq G$. Fix arbitrary $\alpha,\beta\in\lambda$. Then there exists subsets $L,R\in G$ such that $A\cap G_{\alpha,\alpha_0}=L_{\alpha,\alpha_0}$ and $A\cap G_{\beta_0,\beta}=R_{\beta_0,\beta}$. It is obviously that $L_{\alpha,\alpha_0}$ and $R_{\beta_0,\beta}$ are dense subsets of $G_{\alpha,\alpha_0}$ and $G_{\beta_0,\beta}$, respectively. This implies that $L$ and $R$ are dense subsets of $G$. Then by Lemma~\ref{lemma-2.2} we have that
\begin{equation*}
    G_{\alpha,\beta}=(L\cdot V\cdot R)_{\alpha,\beta}=L_{\alpha,\alpha_0}\cdot V_{\alpha_0,\beta_0}\cdot R_{\beta_0,\beta} =(A\cap G_{\alpha,\alpha_0})\cdot U\cdot(A\cap G_{\beta_0,\beta})\subseteq A\cdot U\cdot A.
\end{equation*}
This completes the proof of the lemma.
\end{proof}

Lemma~\ref{lemma-2.2} implies the following:

\begin{proposition}\label{proposition-2.4}
Let $U$ be an open non-empty subset of a topological inverse Brandt semigroup $B_1(G)$ such that $U\neq\{0\}$. Then for every dense subset $A$ of $B_1(G)$ the following statements hold:
\begin{itemize}
    \item[$(i)$] $A\cdot U\cdot A=B_1(G)$ in the case when $0$ is an isolated point in $B_1(G)$;

    \item[$(ii)$] $\left(A\cup\{0\}\right)\cdot U\cdot \left(A\cup\{0\}\right)=B_1(G)$ in the case when $0$ is a non-isolated point in $B_1(G)$.
\end{itemize}
\end{proposition}

Lemma~\ref{lemma-2.3} and Proposition~\ref{proposition-2.4} imply the following proposition:

\begin{proposition}\label{proposition-2.5}
Let $S$ be a Hausdorff primitive inverse topological semigroup such that  $S$ be an orthogonal sum of the family $\{B_{\lambda_{i}}(G_i)\}_{i\in\mathscr{I}}$ of topological Brandt semigroups with zeros. Let $|\mathscr{I}|>1$ and $U$ be an open non-empty subset of $S$ such that $\left(U\cap B_{\lambda_{i}}(G_i)\right)\setminus\{0\}\neq\varnothing$ for any $i\in\mathscr{I}$. Then $A\cdot U\cdot A=S$ for every dense subset $A$ of $S$.
\end{proposition}

\begin{remark}\label{remark-2.6}
Since by Theorem~II.4.3 of~\cite{Petrich1984} a primitive inverse semigroup $S$ is the orthogonal sum of a family of Brandt semigroups, i.e., $S$ is an orthogonal sum $\sum_{i\in\mathscr{I}}B_{\lambda_{i}}(G_i)$ of Brandt $\lambda_i$-extensions $B_{\lambda_i}(G_i)$ of groups $G_i$, we have that Proposition~12 from \cite{BerezovskiGutikPavlyk2010} describes a base of the topology at any non-zero element of $S$.
\end{remark}

Later by $\mathfrak{TISG}$ we denote the category of topological inverse semigroups, where $\mathbf{Ob}\mathfrak{TISG}$ are all topological inverse semigroups and $\mathbf{Mor}\mathfrak{TISG}$ are homomorphisms between topological inverse semigroups.

The following theorem describes the structure of primitive pseudocompact topological inverse semigroups.

\begin{theorem}\label{theorem-2.7}
Every primitive Hausdorff pseudocompact topological inverse semigroup $S$ is topologically isomorphic to the orthogonal sum $\sum_{i\in\mathscr{I}}B_{\lambda_{i}}(G_i)$ of topological Brandt $\lambda_i$-extensions $B_{\lambda_i}(G_i)$ of pseudocompact topological groups $G_i$ in the category $\mathfrak{TISG}$ for some finite cardinals $\lambda_i\geqslant 1$. Moreover the family
\begin{equation}\label{eq-2.1}
 \mathscr{B}(0)=\left\{S\setminus\big(B_{\lambda_{i_1}}(G_{i_1})\cup
 B_{\lambda_{i_2}}(G_{i_2})\cup\cdots\cup B_{\lambda_{i_n}}(G_{i_n})\big)^*\colon i_1,
 i_2,\ldots,i_n\in\mathscr{I}, n\in\mathbb{N}\right\}
\end{equation}
determines a base of the topology at zero $0$ of $S$.
\end{theorem}

\begin{proof}
By Theorem~II.4.3 of~\cite{Petrich1984} the semigroup $S$ is an orthogonal sum of Brandt semigroups and hence $S$ is isomorphic to the orthogonal sum $\sum_{i\in\mathscr{I}}B_{\lambda_{i}}(G_i)$ of Brandt $\lambda_i$-extensions $B_{\lambda_i}(G_i)$ of groups $G_i$. We fix any $i_0\in\mathscr{I}$. Since $S$ is a topological inverse semigroup, Proposition~II.2~\cite{EberhartSelden1969} implies that $B_{\lambda_{i_0}}(G_{i_0})$ is a topological inverse semigroup. By Proposition~\ref{proposition-2.1}, $B_{\lambda_{i_0}}(G_{i_0})$ is a pseudocompact topological Brandt $\lambda_i$-extension of pseudocompact  topological group $G_{i_0}$ in the category $\mathfrak{TISG}$ for some finite cardinal $\lambda_{i_0}\geqslant 1$. This completes the proof of the first assertion of the theorem.

The second statement of the theorem is trivial in the case when the set of indices $\mathscr{I}$ is finite. Hence later we assume that the set $\mathscr{I}$ is infinite.

Suppose on the contrary that $\mathscr{B}(0)$ is not a base at zero $0$ of $S$. Then, there exists an open neighbourhood $U(0)$ of zero $0$ such that $U(0)\bigcup\big(B_{\lambda_{i_1}}(G_{i_1})\cup B_{\lambda_{i_2}}(G_{i_2})\cup\cdots\cup B_{\lambda_{i_n}}(G_{i_n})\big)^*\neq S$ for finitely many indices $i_1, i_2,\ldots,i_n\in\mathscr{I}$. Let $V(0)\subseteq U(0)$ be an open neighbourhood of $0$ in $S$ such that $V(0)\cdot V(0)\cdot V(0)\subseteq U(0)$. Then we have that $V(0)\bigcup\big(B_{\lambda_{i_1}}(G_{i_1})\cup B_{\lambda_{i_2}}(G_{i_2})\cup\cdots\cup B_{\lambda_{i_n}}(G_{i_n})\big)^*\neq S$. We state that there exist a sequence of distinct points $\{x_k\}_{k\in\mathbb{N}}$
of the semigroup $S$ and a sequence of open subsets $\{U(x_k)\}_{k\in\mathbb{N}}$ of $S$ such that the following conditions hold:
\begin{itemize}
  \item[$(i)$] $x_k\in U(x_k)\subseteq B_{\lambda_{i_k}}(G_{i_k})$ for some $i_k\in\mathscr{I}$;
  \item[$(ii)$] if $x_{k_1},x_{k_2}\in B_{\lambda_{i_k}}(G_{i_k})$ for some $i_k\in\mathscr{I}$, then $k_1=k_2$;
  \item[$(iii)$] $\bigcup_{k\in\mathbb{N}}U(x_k)\subseteq S\setminus V(0)$.
\end{itemize}
Otherwise we have that $V(0)$ is a dense subset of the subspace
\begin{equation*}
 S^{\prime}=S\setminus \bigcup\big(B_{\lambda_{i_1}}(G_{i_1})\cup B_{\lambda_{i_2}}(G_{i_2})\cup\cdots\cup B_{\lambda_{i_n}}(G_{i_n}))\big)^*,
\end{equation*}
for some positive integer $n$.
Since $S^{\,\prime}$ with induced operation from $S$ is a primitive inverse semigroup Proposition~\ref{proposition-2.5} implies that $V(0)\cdot V(0)\cdot V(0)=S^{\,\prime}$ which contradicts the choice of the neighbourhood $U(0)$. The obtained contradiction implies that there exists finitely many indexes $i_1, i_2,\ldots,i_n,\ldots,i_m\in \mathscr{I}$ where $m>n$ such that
\begin{equation*}
U(0)\cup\big(B_{\lambda_{i_1}}(G_{i_1})\cup B_{\lambda_{i_2}}(G_{i_2})\cup\cdots\cup B_{\lambda_{i_n}}(G_{i_n})\cup\cdots\cup B_{\lambda_{i_m}}(G_{i_n})\big)^*=S.
\end{equation*}
This completes the proof of the theorem.
\end{proof}

\begin{proposition}\label{proposition-2.8}
Let $S$ be a primitive Hausdorff pseudocompact topological inverse semigroup which is topologically isomorphic to the orthogonal sum $\sum_{i\in\mathscr{I}}B_{\lambda_{i}}(G_i)$ of topological Brandt $\lambda_i$-extensions $B_{\lambda_i}(G_i)$ of topological groups $G_i$ in the category $\mathfrak{TISG}$ for some cardinals $\lambda_i\geqslant 1$. Then the following conditions hold:
\begin{itemize}
  \item[$(i)$] the space $S$ is Tychonoff if and only if for every $i\in\mathscr{I}$ the space of the topological group $G_i$ is Tychonoff, i.e., $G_i$ is a $T_0$-space;
  \item[$(ii)$] the space $S$ is normal if and only if for every $i\in\mathscr{I}$ the space of the topological group $G_i$ is normal.
\end{itemize}
\end{proposition}

\begin{proof}
We observe that the $T_0$-topological space of a topological group is Tychonoff (see Theorem~2.6.4 in \cite{HewittRoss1963}).

$(i)$ Implication $(\Rightarrow)$ follows from Theorem~2.1.6 of \cite{Engelking1989}.

$(\Leftarrow)$ Suppose that for every $i\in\mathscr{I}$ the space of the topological group $G_i$ is Tychonoff. We fix an arbitrary element $x\in S$. First we consider the case when $x\neq 0$. Then there exists an non-zero $\mathscr{H}$-class $H$ which contains $x$. By Proposition~12 from \cite{BerezovskiGutikPavlyk2010} there exists $i\in\mathscr{I}$ such that the topological space $H$ is homeomorphic to the topological group $G_i$. Then by Proposition~1.5.8 from \cite{Engelking1989} for every open neighbourhood $U(x)$ of $x$ in $H$ there exists a continuous map $f\colon H\rightarrow[0,1]$ such that $f(x)=0$ and $f(y)=1$ for all $y\in H\setminus U(x)$. We define the map $\widetilde{f}\colon S\rightarrow[0,1]$ in the following way:
\begin{equation*}
    \widetilde{f}(y)=
\left\{
  \begin{array}{ll}
    f(y), & \hbox{if~~} y\in H;\\
    1, & \hbox{if~~} y\in S\setminus H.
  \end{array}
\right.
\end{equation*}
Since by Proposition~12 from \cite{BerezovskiGutikPavlyk2010} every non-zero $\mathscr{H}$-class is an open-and-closed subset of $S$ we conclude that such defined map $\widetilde{f}\colon S\rightarrow[0,1]$ is continuous.

Suppose that $x=0$. We fix an arbitrary $U(0)=S\setminus\big(B_{\lambda_{i_1}}(G_{i_1})\cup
 B_{\lambda_{i_2}}(G_{i_2})\cup\cdots\cup B_{\lambda_{i_n}}(G_{i_n})\big)^*\in\mathscr{B}(0)$. Then by Proposition~12 from \cite{BerezovskiGutikPavlyk2010}, $U(0)$ is an open-and-closed subset of $S$. Thus we have that the map $f\colon S\rightarrow[0,1]$ defined by the formula
\begin{equation*}
    \widetilde{f}(y)=
\left\{
  \begin{array}{ll}
    0, & \hbox{if~~} y\in U(0);\\
    1, & \hbox{if~~} y\in S\setminus U(0),
  \end{array}
\right.
\end{equation*}
is continuous, and hence by Proposition~1.5.8 from \cite{Engelking1989} the space $S$ is Tychonoff.

Next we shall prove statement $(ii)$.

$(\Rightarrow)$ Suppose that $S$ is a normal space. By Lemma~9 of \cite{BerezovskiGutikPavlyk2010} we have that every $\mathscr{H}$-class of $S$ is a closed subset of $S$. Then by Theorem~2.1.6 from \cite{Engelking1989} we have that  every $\mathscr{H}$-class of $S$ is a normal subspace of $S$ and hence Definition~\ref{def1} and Proposition~12 of \cite{BerezovskiGutikPavlyk2010} imply that for every $i\in\mathscr{I}$ the space of the topological group $G_i$ is normal.

$(\Leftarrow)$ Suppose that for every $i\in\mathscr{I}$ the space of the topological group $G_i$ is normal. Let $F_1$ and $F_2$ be arbitrary closed disjoint subsets of $S$.

At first we consider the case when $0\notin F_1\cup F_2$. Then there exists an open neighbourhood $U(0)$ of zero in $S$ such that $F_1\cup F_2\subseteq S\setminus U(0)$, i.e., there exist finitely many $i_1, i_2,\ldots,i_n\in\mathscr{I}$ such that
\begin{equation*}
F_1\cup F_2\subseteq \left(B_{\lambda_{i_1}}(G_{i_1})\cup
B_{\lambda_{i_2}}(G_{i_2}) \cup\cdots\cup B_{\lambda_{i_n}}(G_{i_n})\right)\setminus\{0\}.
\end{equation*}
By Corollary~8 of \cite{BerezovskiGutikPavlyk2010} every non-zero $\mathscr{H}$-class of $S$ is open subset in $S$, and hence we get that the subspace $\left(B_{\lambda_{i_1}}(G_{i_1})\cup
B_{\lambda_{i_2}}(G_{i_2})\cup\cdots\cup B_{\lambda_{i_n}}(G_{i_n})\right)\setminus\{0\}$ of $S$ is a topological sum of some non-zero $\mathscr{H}$-classes of $S$, and hence it is an open subspace of $S$. Then by Theorem~2.2.7 from \cite{Engelking1989} we have that $\left(B_{\lambda_{i_1}}(G_{i_1})\cup
B_{\lambda_{i_2}}(G_{i_2})\cup\cdots\cup B_{\lambda_{i_n}}(G_{i_n})\right)\setminus\{0\}$ is a normal space. Therefore, there exist disjoint open neighbourhoods $V(F_1)$ and $V(F_2)$ of $F_1$ and $F_2$ in $\left(B_{\lambda_{i_1}}(G_{i_1})\cup
B_{\lambda_{i_2}}(G_{i_2})\cup\cdots\cup B_{\lambda_{i_n}}(G_{i_n})\right)\setminus\{0\}$, and hence in $S$, respectively.

Suppose that $0\in F_1\cup F_2$. Without loss of generality we can assume that $0\in F_1$. Then there exist finitely many $i_1, i_2,\ldots,i_n\in\mathscr{I}$ such that
\begin{equation*}
F_2\subseteq \left(B_{\lambda_{i_1}}(G_{i_1})\cup
B_{\lambda_{i_2}}(G_{i_2}) \cup\cdots\cup B_{\lambda_{i_n}}(G_{i_n})\right)\setminus\{0\}.
\end{equation*}
The assumption of the proposition implies that the set $\left(B_{\lambda_{i_1}}(G_{i_1})\cup
B_{\lambda_{i_2}}(G_{i_2}) \cup\cdots\cup B_{\lambda_{i_n}}(G_{i_n})\right)\setminus\{0\}$ is closed in $S$ and hence
\begin{equation*}
\widetilde{F}_1=F_1\cap\left(\left(B_{\lambda_{i_1}}(G_{i_1})\cup
B_{\lambda_{i_2}}(G_{i_2}) \cup\cdots\cup B_{\lambda_{i_n}}(G_{i_n})\right)\setminus\{0\}\right)
\end{equation*}
is a closed subset of $S$, too. Then the previous arguments of the proof imply that \begin{equation*}
\left(B_{\lambda_{i_1}}(G_{i_1})\cup
B_{\lambda_{i_2}}(G_{i_2})\cup\cdots\cup B_{\lambda_{i_n}}(G_{i_n})\right)\setminus\{0\}
\end{equation*}
is a normal space, and hence there exist open disjoint neighbourhoods $W(\widetilde{F}_1)$ and $U(F_2)$ of the closed sets $\widetilde{F}_1$ and $F_2$ in $\left(B_{\lambda_{i_1}}(G_{i_1})\cup
B_{\lambda_{i_2}}(G_{i_2})\cup\cdots\cup B_{\lambda_{i_n}}(G_{i_n})\right)\setminus\{0\}$, and hence in $S$, respectively. We put
\begin{equation*}
U(F_1)=S\setminus\left(B_{\lambda_{i_1}}(G_{i_1})\cup
B_{\lambda_{i_2}}(G_{i_2}) \cup\cdots\cup B_{\lambda_{i_n}}(G_{i_n})\right)^*\cup W(\widetilde{F}_1).
\end{equation*}
Then we have that $U(F_1)$ and $U(F_2)$ are open disjoint neighbourhoods of $F_1$ and $F_2$ in $S$, respectively. This completes the proof of statement $(ii)$.
\end{proof}

Theorem~\ref{theorem-2.7} and Proposition~\ref{proposition-2.8} imply the following:

\begin{corollary}\label{corollary-2.9}
Every primitive Hausdorff pseudocompact topological inverse semigroup $S$ is a Tychonoff topological space. Moreover the topological space of $S$ is normal if and only if every maximal subgroup of $S$ is a normal subspace.
\end{corollary}

By Theorem~3.10.21 from \cite{Engelking1989} every normal pseudocompact space is countably compact, and hence Corollary~\ref{corollary-2.9} implies the following:

\begin{corollary}\label{corollary-2.10}
Every primitive Hausdorff pseudocompact topological inverse semigroup $S$ such that every maximal subgroup of $S$ is a normal subspace in $S$ is countably compact.
\end{corollary}

\begin{proposition}\label{proposition-2.11}
Every primitive pseudocompact topological inverse semigroup $S$ is a continuous (non-homomorphic) image of the product $\widetilde{E}_S\times G_S$, where $\widetilde{E}_S$ is a compact semilattice and $G_S$ is a pseudocompact topological group.
\end{proposition}

\begin{proof}
By Theorem~\ref{theorem-2.7} the topological semigroup $S$ is topologically isomorphic to the orthogonal sum $\sum_{i\in\mathscr{I}}B_{\lambda_{i}}(G_i)$ of topological Brandt $\lambda_i$-extensions $B_{\lambda_i}(G_i)$ of pseudocompact topological groups $G_i$ in the category $\mathfrak{TISG}$ for some finite cardinals $\lambda_i\geqslant 1$ and the family defined by formula (\ref{eq-2.1}) determines the base of the topology at zero of $S$.

Fix an arbitrary $i\in\mathscr{I}$. Then by Proposition~\ref{proposition-2.1}$(iv)$ the set $E(B_{\lambda_i}(G_i))$ is finite. Suppose that $|E(B_{\lambda_i}(G_i))|=n_i+1$ for some integer $n_i$. Then we have that $\lambda_i=n_i\geqslant1$. On the set $E_i=(\lambda_i\times\lambda_i)\cup\{0\}$, where $0\notin\lambda_i\times\lambda_i$ we define the binary operation in the following way
\begin{equation*}
(\alpha,\beta)\cdot(\gamma,\delta)=
    \left\{
       \begin{array}{cl}
         (\alpha,\beta), & \hbox{if~} (\alpha,\beta)=(\gamma,\delta);\\
         0, & \hbox{otherwise,}
       \end{array}
     \right.
\end{equation*}
and $0\cdot(\alpha,\beta)=(\alpha,\beta)\cdot 0=0\cdot 0=0$ for all $\alpha,\beta,\gamma,\delta\in\lambda_i$. Simple verifications show that $E_i$ with such defined operation is a semilattice and every non-zero idempotent of $E_i$ is primitive.

By $\widetilde{E}_S$ we denote the orthogonal sum $\sum_{i\in\mathscr{I}}E_i$. It is obvious that $\widetilde{E}_S$ is a semilattice and every non-zero idempotent of $\widetilde{E}_S$ is primitive. We determine on $\widetilde{E}_S$ the topology of the Alexandroff one-point compactification $\tau_A$: all non-zero idempotents of $\widetilde{E}_S$ are isolated points in $\widetilde{E}_S$ and the family
\begin{equation*}
    \mathscr{B}(0)=\big\{U\colon U\ni 0 \;\hbox{ and } \; \widetilde{E}_S\setminus U \; \hbox{ is finite}\big\}
\end{equation*}
is the base of the topology $\tau_A$ at zero $0\in\widetilde{E}_S$. Simple verifications show that $\widetilde{E}_S$ with the topology $\tau_A$ is a Hausdorff compact topological semilattice. Later we denote $(\widetilde{E}_S,\tau_A)$ by $\widetilde{E}_S$.

Let $G_S=\prod_{i\in\mathscr{I}}G_i$ be the direct product of pseudocompact groups $G_i$, $i\in\mathscr{I}$, with the Tychonoff topology. Then by Comfort--Ross Theorem (see Theorem~1.4 in \cite{ComfortRoss1966}) we get that $G_S$ is a pseudocompact topological group. Also by Corollary~3.10.27 from \cite{Engelking1989} we have that the product $\widetilde{E}_S\times G_S$ is a pseudocompact space.

Later for every $i\in\mathscr{I}$ by $\pi_i\colon G_S=\prod_{i\in\mathscr{I}}G_i\rightarrow G_i$ we denote the projection on the $i$-th factor.

Now, for every $i\in\mathscr{I}$ we define the map $f_i\colon E_i\times G_S\rightarrow B_{\lambda_i}(G_i)$  by the formulae $f_i((\alpha,\beta),g)=(\alpha,\pi_i(g),\beta)$ and $f_i(0,g)=0_i$ is zero of the semigroup $B_{\lambda_i}(G_i)$, and put $f=\bigcup_{i\in\mathscr{I}}f_i$. It is obvious that the map $f\colon \widetilde{E}_S\times G_S\rightarrow S$ is well defined. The definition of the topology $\tau_A$ on $\widetilde{E}_S$ implies that for every $((\alpha,\beta),g)\in E_i\times G_i\subseteq \widetilde{E}_S\times G_i$  the set $\{(\alpha,\beta)\}\times G_i$ is open in $\widetilde{E}_S\times G_S$ and hence the map $f$ is continuous at the point $((\alpha,\beta),g)$. Also for every $U(0)=S\setminus\big(B_{\lambda_{i_1}}(G_{i_1})\cup
B_{\lambda_{i_2}}(G_{i_2})\cup\cdots\cup B_{\lambda_{i_n}}(G_{i_n})\big)^*$ the set $f^{-1}(U(0))=\big(\widetilde{E}_S\setminus\left((\lambda_{i_1}\times\lambda_{i_1}) \cup\ldots\cup(\lambda_{i_n}\times\lambda_{i_n})\right)\big)\times G_S$ is open in $\widetilde{E}_S\times G_S$, and hence the map $f$ is continuous.
\end{proof}

The following theorem is an analogue of Comfort--Ross Theorem for primitive pseudocompact topological inverse semigroup.

\begin{theorem}\label{theorem-2.12}
Let $\{S_i\colon i\in\mathscr{J}\}$ be a non-empty family of primitive Hausdorff pseudocompact topological inverse semigroups. Then the direct product $\prod_{j\in\mathscr{J}}S_j$ with the Tychonoff topology is a pseudocompact topological inverse semigroup.
\end{theorem}

\begin{proof}
Since the direct product of the non-empty family topological inverse semigroups is a topological inverse semigroup, it is sufficient to show that the space $\prod_{j\in\mathscr{J}}S_j$ is pseudocompact. Let $\widetilde{E}_{S_j}$, ${G}_{S_j}$, and $f_j\colon \widetilde{E}_{S_j}\times {G}_{S_j}\rightarrow S_j$ be the semilattice, the group and the map, respectively, defined in the proof of Proposition~\ref{proposition-2.11} for any $j\in\mathscr{J}$. Since the space $\prod_{j\in\mathscr{J}}\left(\widetilde{E}_{S_j}\times {G}_{S_j}\right)$ is homeomorphic to the product $\prod_{j\in\mathscr{J}}\widetilde{E}_{S_j}\times \prod_{j\in\mathscr{J}}{G}_{S_j}$ we conclude that by Theorem~3.2.4, Corollary~~3.10.27 from \cite{Engelking1989} and Theorem~1.4 from \cite{ComfortRoss1966} the space $\prod_{j\in\mathscr{J}}\left(\widetilde{E}_{S_j}\times {G}_{S_j}\right)$ is pseudocompact. Now, since the map $\prod_{j\in\mathscr{J}}f_j\colon \prod_{j\in\mathscr{J}}\left(\widetilde{E}_{S_j}\times {G}_{S_j}\right)\rightarrow \prod_{j\in\mathscr{J}}S_j$ is continuous we have that $\prod_{j\in\mathscr{J}}S_j$ is a pseudocompact topological space.
\end{proof}

Theorem~\ref{theorem-2.12} implies the following corollary:

\begin{corollary}\label{corollary-2.13}
Let $\{S_i\colon i\in\mathscr{J}\}$ be a non-empty family of Brandt Hausdorff pseudocompact topological inverse semigroups. Then the direct product $\prod_{j\in\mathscr{J}}S_j$ with the Tychonoff topology is a pseudocompact topological inverse semigroup.
\end{corollary}

\begin{remark}\label{remark-2.14}
E. K. van Douwen \cite{van_Douwen1980} showed that Martin's Axiom implies the existence of two countably compact groups (without non-trivial convergent sequences) such that theirs product is not countably compact. Hart and van Mill \cite{Hart-van_Mill1991} showed that Martin's Axiom for countable posets implies the existence of a countably compact group which square is not countably compact. Tomita in \cite{Tomita1996} showed that under $MA_{\textrm{countable}}$ for each positive integer $k$ there exists a group which $k$-th power is countably compact but its $2k$-th power is not countably compact. In particular, there was proved that for each positive integer $k$ there exists $l=k,\ldots, 2k-1$ and a group which $l$-th power is not countably compact. In \cite{Tomita1999} Tomita constructed a topological group under $MA_{\textrm{countable}}$ which square is countably compact but its cube is not. Also, Tomita in \cite{Tomita2005} showed that the existence of $2^\mathfrak{c}$ mutually incorparable selective ultrafilters and $2^\mathfrak{c}=2^{2^\mathfrak{c}}$ implies that there exists a topological group $G$  such that $G^\gamma$ is countably compact for all cardinals $\gamma<\kappa$, but $G^\kappa$ is not countably compact for every cardinal $\kappa\leqslant 2^\mathfrak{c}$. Using these results and the construction of finite topological Brandt $\lambda^0$-extensions proposed in \cite{GutikRepovs2010} we may show that statements similar to aforementioned hold for Hausdorff countably compact Brandt topological inverse semigroups and hence for Hausdorff countably compact primitive topological inverse semigroups.
\end{remark}

\section{The Stone-\v{C}ech compactification of a pseudocompact primitive topological inverse semigroup}

Let a Tychonoff topological space $X$ be a topological sum of subspaces $A$ and $B$, i.e., $X=A\bigoplus B$. It is obvious that every continuous map $f\colon A\rightarrow K$ from $A$ into a compact space $K$ (resp., $f\colon B\rightarrow K$ from $B$ into a compact space $K$) extends to a continuous map $\widehat{f}\colon X\rightarrow K$. This implies the following proposition:

\begin{proposition}\label{proposition-3.1}
If a Tychonoff topological space $X$ is a topological sum of subspaces $A$ and $B$, then $\beta X$ is equivalent to $\beta A\bigoplus\beta B$.
\end{proposition}

The following theorem describes the structure of the Stone-\v{C}ech compactification of a primitive pseudocompact topological inverse semigroup.

\begin{theorem}\label{theorem-3.2}
Let $S$ be a primitive pseudocompact topological inverse semigroup. Then the Stone-\v{C}ech compactification of $S$ admits a structure of primitive topological inverse semigroup with respect to which the inclusion mapping of $S$ into $\beta{S}$ is a topological isomorphism. Moreover, $\beta{S}$ is topologically isomorphic to the orthogonal sum $\sum_{i\in\mathscr{I}}B_{\lambda_{i}}(\beta G_i)$ of topological Brandt $\lambda_i$-extensions $B_{\lambda_i}(\beta G_i)$ of compact topological groups $\beta G_i$ in the category $\mathfrak{TISG}$ for some finite cardinals $\lambda_i\geqslant 1$.
\end{theorem}

\begin{proof}
By Theorem~\ref{theorem-2.7}, every primitive pseudocompact topological inverse semigroup $S$ is topologically isomorphic to the orthogonal sum $\sum_{i\in\mathscr{I}}B_{\lambda_{i}}(G_i)$ of topological Brandt $\lambda_i$-extensions $B_{\lambda_i}(G_i)$ of pseudocompact topological groups $G_i$ in the category $\mathfrak{TISG}$ for some finite cardinals $\lambda_i\geqslant 1$, such that any non-zero $\mathscr{H}$-class of $S$ is an open-and-closed subset of $S$, and the family $\mathscr{B}(0)$ defined by formula (\ref{eq-2.1}) determines a base of the topology at zero $0$ of $S$.

By Theorem~\ref{theorem-2.12}, $S\times S$ is a pseudocompact topological space. Now by Theorem~1 of \cite{Glicksberg1959}, we have that $\beta(S\times S)$ is equivalent to $\beta S\times \beta S$, and hence by Theorem~1.3~\cite{BanakhDimitrova2010}, $S$ is a subsemigroup of the compact topological semigroup $\beta S$.

By Proposition~\ref{proposition-3.1} for every non-zero $\mathscr{H}$-class $(G_i)_{k,l}$, $k,l\in\lambda_i$, we have that $\operatorname{cl}_{\beta S}((G_i)_{k,l})$ is equivalent to $\beta(G_i)_{k,l}$, and hence it is equivalent to $\beta G_i$. Therefore we get that $\sum_{i\in\mathscr{I}}B_{\lambda_{i}}(G_i)\subseteq \beta S$. Suppose that $\sum_{i\in\mathscr{I}}B_{\lambda_{i}}(G_i)\neq \beta S$. We fix an arbitrary $x\in\beta S\setminus \sum_{i\in\mathscr{I}}B_{\lambda_{i}}(G_i)$. Then Hausdorffness of $\beta S$ implies that there exist open neighbourhoods $V(x)$ and $V(0)$ of the points $x$ and $0$ in $\beta S$, respectively, and there exist finitely many $i_1,\ldots,i_n\in\mathscr{I}$ such that $V(0)\cap\beta S\supseteq S\setminus\big(B_{\lambda_{i_1}}(G_{i_1})\cup B_{\lambda_{i_2}}(G_{i_2}) \cup \cdots \cup B_{\lambda_{i_n}}(G_{i_n})\big)^*$. Then we have that
\begin{equation*}
V(x)\cap S \subseteq \big(B_{\lambda_{i_1}}(G_{i_1})\cup B_{\lambda_{i_2}}(G_{i_2})\cup\cdots \cup B_{\lambda_{i_n}}(G_{i_n})\big)^*\subseteq \big(B_{\lambda_{i_1}}(\beta G_{i_1})\cup B_{\lambda_{i_2}}(\beta G_{i_2})\cup\cdots \cup B_{\lambda_{i_n}}(\beta G_{i_n})\big)^*,
\end{equation*}
and since by Theorem~\ref{theorem-2.7}, $\lambda_i$ is finite for every $i\in\mathscr{I}$, we get a contradiction with the initial assumption. This completes the proof of the theorem.
\end{proof}

Theorem~\ref{theorem-3.2} implies the following:

\begin{corollary}\label{corollary-3.3}
Let $S$ be a primitive countably compact topological inverse semigroup. Then the Stone-\v{C}ech compactification of $S$ admits a structure of primitive topological inverse semigroup with respect to which the inclusion mapping of $S$ into $\beta{S}$ is a topological isomorphism.
\end{corollary}

\begin{remark}\label{remark-3.4}
Theorem~\ref{theorem-3.2} and Corollary~\ref{corollary-3.3} give the positive answer to the Question~1, which we posed in \cite{BerezovskiGutikPavlyk2010}.
\end{remark}

We define the series of categories as follows:
\begin{itemize}
    \item[$(i)$] Let     $\operatorname{\textbf{Ob}}(\mathscr{B}^*(\mathscr{CC\!T\!G}))$ be all Hausdorff  $0$-simple countably compact topological inverse semigroups;
    \item[] Let     $\operatorname{\textbf{Ob}}(\mathscr{B}^*(\mathscr{PC\!T\!G}))$ be all Hausdorff pseudocompact topological inverse Brandt semigroups;
    \item[] Let $\operatorname{\textbf{Ob}}(\mathscr{P\!PC\!T\!G})$ be all primitive Hausdorff pseudocompact topological inverse semigroups;
    \item[] Let $\operatorname{\textbf{Ob}}(\mathscr{PCC\!T\!G})$ be all primitive Hausdorff pseudocompact topological inverse semigroups;
   \item[$(ii)$]  Let $\operatorname{\textbf{Mor}}(\mathscr{B}^*(\mathscr{CC\!T\!G})))$,
       $\operatorname{\textbf{Mor}}(\mathscr{B}^*(\mathscr{PC\!T\!G}))$,
       $\operatorname{\textbf{Mor}}(\mathscr{P\!PC\!T\!G})$, and
       $\operatorname{\textbf{Mor}}(\mathscr{PCC\!T\!G})$ be continuous homomorphisms of of corresponding topological inverse semigroups.
\end{itemize}

Comfort and Ross \cite{ComfortRoss1966} proved that the
Stone-\v{C}ech compactification of a pseudocompact topological
group is a topological group. Therefore the functor of the
Stone-\v{C}ech compactification \textbf{$\beta$} from the category
of pseudocompact topological groups back into itself determines a
monad. Similar result Gutik and Repov\v{s} proved in \cite{GutikRepovs2010} for the category  of all Hausdorff $0$-simple countably compact topological inverse semigroups $\mathscr{B}^*(\mathscr{CC\!T\!G})$. In the our case by Theorem~\ref{theorem-3.2} and Corollary~\ref{corollary-3.3} we get the same:

\begin{corollary}\label{corollary-3.5}
The functor of the Stone-\v{C}ech compactification $\beta\colon
\mathscr{B}^*(\mathscr{CC\!T\!G}) \rightarrow
\mathscr{B}^*(\mathscr{CC\!T\!G})$ $($resp., $\beta\colon
\mathscr{B}^*(\mathscr{PC\!T\!G}) \rightarrow
\mathscr{B}^*(\mathscr{PC\!T\!G}), \; \beta\colon
\mathscr{P\!PC\!T\!G} \rightarrow \mathscr{P\!PC\!T\!G}, \; \beta\colon
\mathscr{PCC\!T\!G} \rightarrow \mathscr{PCC\!T\!G} )$ determines a monad.
\end{corollary}

\section*{Acknowledgements}

This research was carried out with the support of the Estonian Science Foundation and co-funded by Marie Curie Action, grant ERMOS36.

We thank the referees and the editor for several comments and remarks.


\end{document}